\documentclass{amsart}
\usepackage{amstext,amsfonts,amsmath,amsthm,amssymb,amscd}

\begin{document}

\title{ A note on the plane Jacobian conjecture}
\author{Nguyen Van Chau}
\address{Institute of Mathematics, 18 Hoang Quoc Viet, 10307 Hanoi,Vietnam}
\email{nvchau@math.ac.vn}
\thanks{This research is supposed by NAFOSTED, Vietnam}
\date{}
\maketitle

\def\C{\mathbb{C}}
\def\P{\mathbb{P}}
\def\DD{\mathcal{D}}

\newtheorem{fact}{Fact}
\begin{abstract}
It is shown that every polynomial function
$P:\mathbb{C}^2\longrightarrow\mathbb{C}$ with irreducible fibres
of same a genus is a coordinate. In consequence, there does not
exist counterexamples $F=(P,Q)$ to the Jacobian conjecture such
that all fibres of $P$ are irreducible curves of same a genus.



\end{abstract}

\newtheorem{theorem}{Theorem}
\newtheorem{lemma}{Lemma}
\newtheorem{proposition}{Proposition}


\vskip0.3cm\noindent{\bf 1.} The Jacobian conjecture (JC), posed
first in 1939 by Ott-Heinrich Keller \cite{Keller} and still
opened, asserts that {\it every polynomial maps
$F=(P,Q):\mathbb{C}^2 \longrightarrow \mathbb{C}^2$ with non-zero
constant Jacobian $J(P,Q):=P_xQ_y-P_yQ_x\equiv c\neq 0$ is
invertible}. Its history, many references, and some partial
results, can be found in \cite{Bass} and \cite{Essen-book}.

Since 1979 in his work \cite{Razar} concerning with the plane case
of (JC) Razar had found the following:
\begin{theorem}[\cite{Razar}]\label{Razar} A non-zero constant Jacobian polynomial
map $F=(P,Q)$ is invertible if
all fibres of $P$ are irreducible rational curves.
\end{theorem}
In other words, it claim that there does not exists
counterexamples $F=(P,Q)$ to (JC) such that all fibres of $P$ are
irreducible rational curves. In attempting to understand the
nature of the plane Jacobian conjecture, Razar 's result had been
reproved by Heitmann \cite{Heitmann}, L\^e and Weber \cite{LeWe2},
Friedland \cite{Friedland}, Nemethi and Sigray
\cite{NemethiSigray} in several different algebraic and
algebro-geometric approaches. It was also proved in
\cite{NemethiSigray} that there does note exists counterexamples
$F=(P,Q)$ to (JC) such that all fibres of $P$ are irreducible
elliptic curves.  Following \cite{NeumannNorbudy}, a polynomial
function $f :\C^2\longrightarrow \C$ is called {\it coordinate} if
there is $g\in \C[x,y]$ such that $(f,g):\C^2\longrightarrow \C^2$
is a polynomial automorphism. In fact, Vistoli \cite{Vistoli} and
Neumann and Norbury \cite{NeumannNorbudy} obtained the following,
which implies Theorem \ref{Razar}.

\begin{theorem}\label{ThmNeumann} A polynomial
function $P:\mathbb{C}^2\longrightarrow\mathbb{C}$ with
irreducible rational fibres is a coordinate.
\end{theorem}

This theorem, as mentioned in \cite{NeumannNorbudy},  is implicit
in Myianishi's  earlier work \cite{Miyanishi} (Lemma 1.7). In this
article we would like to present the following improvement of
Theorem \ref{ThmNeumann}.

\begin{theorem}\label{Main} A polynomial
function $P:\mathbb{C}^2\longrightarrow\mathbb{C}$ with
irreducible fibres of same a genus is a coordinate.
\end{theorem}

Here, as usual, we mean the genus of an irreducible algebraic
curve to be the genus of the desingularization of this curve. In
other words, Theorem \ref{Main} says that the genus of the fibres
of any polynomial with irreducible fibres must  vary, except for
coordinates. In view of this theorem there does not exist
counterexamples $F=(P,Q)$ to (JC) such that all fibres of $P$ are
irreducible curves of same a genus. Note that, following Kaliman
\cite{Kaliman}, to prove the plane case of (JC) it is sufficient
to consider only non-zero constant Jacobian polynomial maps
$F=(P,Q)$ in which all fibres of $P$ are irreducible. Further
investigations on how the genus of the fibres of polynomial
functions in $\C^2$ change should be useful in hunting the
solution of the plane Jacobian conjecture.

The proof of Theorem \ref{Main}, presented in Section 2, is a
combination of Theorem \ref{ThmNeumann} and the fact that if the
fibres of $P$ are irreducible curves of same a genus $g$, then
$g=0$ (Lemma 1 and Lemma 2). The proofs of Theorem
\ref{ThmNeumann} in \cite{NeumannNorbudy}, \cite{Miyanishi} and
\cite{Vistoli} use  Abhyankar-Moh-Suzuki Embedding Theorem,
 or  Zariski's main
theorem. In Section 3 we will give another elementary proof of
Theorem \ref{ThmNeumann} by applying Newton-Puiseux Theorem and
the basic properties of the standard resolution of singularities.

\vskip0.3cm\noindent{\bf 2.}  Let $P\in\C[x,y]$ be a given
non-constant polynomial. Regard the plane $\mathbb{C}^2$ as a
subset of the projective plane $\mathbb{P}^2$ and consider $P$ as
a rational morphism $P:\P^2\longrightarrow \P$, which is well
defined everywhere on $\P^2$, except a finite number of points in
the line at infinity $z=0$. By blowing-up we can remove
indeterminacy points of $P$ and obtain a compactification $X$ of
$\C^2$ and a regular extension $p:X\longrightarrow \P$ of $P$ over
$X$. The divisor $\DD:=X\setminus \C^2$ is the union of smooth
rational curves with simple normal crossings. An irreducible
component $D$ of $\DD$ is called {\it horizontal } curve of $P$ if
$p_{|D}$ is a non-constant mapping. Note that the number of
horizontal curves of $P$ does not depend on the regular extension
$p$. Following \cite{Le-kyoto} such a compactification $X$ is
called {\it minimal} if among components of $\DD$ the only the
proper transform of the line at infinity $z=0$ and horizontal
curves may have self-intersections $-1$. A regular extension of
$P$ over a minimal compactification of $\C^2$ can be constructed
by a blowing-up process in which we blow up  only at the
indeterminacy points of $P$ and of the resulted blowing-up
versions of $P$.

Now, suppose  $p:X\longrightarrow \mathbb{P}^1$ is a regular
extension of $P$ over a  minimal compactification $X$ of
$\mathbb{C}^2$. Let us denote by $C_s$ the fiber of $p$ over $s\in
\P^1$ and by $C$ the generic fiber of $p$. Let $\DD_s:=\DD\cap
C_s$ - the portion of $C_s$ contained in $\DD$. The minimality of
the compactification $X$ of $\C^2$ yields the following
usefulness:
\begin{enumerate}
\item[(*)] {\it Every irreducible components in $\DD_s$, $s\in
\C$, must has self-intersections less than $-1$.}
\end{enumerate}

We begin with the following observation.

\begin{lemma}\label{Smooth} Assume that all fibres of $P$ are irreducible curves of same a genus
$g$. Then, the fibres $C_s$, $s\in \C$, are smooth irreducible
curves of same genus $g$.
\end{lemma}

\begin{proof} As usual, we denote by $K_X$ the canonical bundle of the surface
$X$, by $\pi(V)$ the virtual genus of an algebraic curve $V$ in
$X$, and by $g(V)$ the genus of the desingularization of $V$ for
when $V$ is irreducible. By the adjunction formula
$$2\pi(V)-2=K_X.V+V.V,$$ and,  for irreducible
curve $V$, $\pi(V)=g(V)$ if and only if $V$ is smooth.
Furthermore, if $V$ is a fiber of a fibration over $X$, then
$V.V=0$ and hence
$$2\pi(V)-2=K_X.V,$$
(see, for example, in \cite{Griff}). For $s\in \C$ let us denote
by $F_s$ the closure in $X$ of the curve $\{(x,y)\in \mathbb{C}^2:
P(x,y)=s\}$. By assumptions the curves $F_s$ are irreducible and
$$g(F_s)\equiv g, \ s\in \C.\eqno(1)$$

Now, let $s\in \C$ be given. By the adjunction formula (see in
\cite{Griff})
$$2g-2=K_X.C_s.\eqno(2)$$
If $C_s$ is irreducible, we have $C_s=F_s$ and $F_s.F_s=0$. Again
by the adjunction formula we get $2\pi(F_s)-2=K_X.F_s$. Therefore,
by (1) and (2) we obtain $\pi(F_s)=g=g(F_s)$. Thus, $C_s $ is a
smooth irreducible curves of genus $g$. So, to complete the proof
we need to show only that $C_s$ is irreducible. Indeed, assume the
contrary that  $C_s$ is not irreducible. Write
$$C_s=\sum_{i=1}^km_iC_i+nF_s,$$
where $C_i$ are irreducible components of $\DD_s$ with
multiplicity $m_i$ and $n$ is the multiplicity of $F_s$ in $C_s$.
The equality (2) becomes
$$2g-2=\sum_{i=1}^k m_iK_X.C_i+ nK_X.F_s.\eqno(3)$$
Since $C_i$ are smooth irreducible rational curves, $\pi(C_i)=0$.
Furthermore, $C_i.C_i <-1$ by Property (*) and $F_s^2<0$ by
Zariski's lemma (see, for example, \cite{Ven}). Then, applying the
adjunction formula to $C_i$ and $F_s$ we  have
$$K_X.C_i=-(C_i.C_i+2)\geq 0$$
and
$$K_X.F_s=2\pi(F_s)-2-F_s.F_s>2\pi(F_s)-2.$$
From the above estimations and (3)  it follows that
$2g-2>2\pi(F_s)-2$. This is impossible, since $\pi(F_s)\geq
g(F_s)=g$. Hence, $C_s$ must be irreducible.
\end{proof}

\begin{lemma}\label{rational} Let $P$ be as in Lemma 1. Then,
$g=0$ and $P$ has only one horizontal curve.
\end{lemma}

\begin{proof} We will use Suzuki's formula \cite{Suzuki}
$$\sum_{s\in\P^1}\chi(C_s)-\chi(C)=\chi(X)-2\chi(C)\eqno(4)$$
Here, $\chi(V )$ indicates the Euler-Poincare characteristic of
$V$. Let us denote by $m$ the number of irreducible components of
the divisor $\DD$, by $m_\infty$  the number of irreducible
components of $C_\infty$ and by $h$ the number of the horizontal
curves of $P$. Note that $\chi(X)= 2+m$ and
$\chi(C_\infty)=1+m_\infty.$ Furthermore, in view of Lemma
\ref{Smooth} $\chi(C_s)= \chi(C)=2-2g$ for all $ s\in \C$ and $
h=m-m_\infty.$ Now, by the above estimations we have
$$\sum_{s\in
\P^1}\chi(C_s)-\chi(C)=\chi(C_\infty)-\chi(C)=1+m_\infty-(2-2g)$$
and $$\chi(X)-2\chi(C)=2+m-2(2-2g).$$ Then, the equality (4)
becomes
$$1+m_\infty-(2-2g)=2+m-2(2-2g),$$
or equivalent,
$$2g=1-(m-m_\infty)=1-h.$$
Since $g\geq 0$ and $h\geq 1$, it follows that $g=0$ and $h=1$.
\end{proof}

\begin{proof}[Proof of Theorem 3]
Combining Lemma \ref{rational} with Theorem \ref{ThmNeumann}.
\end{proof}

\vskip0.3 cm\noindent{\bf 3. }  The main arguments in the proofs
of Theorem \ref{ThmNeumann} in \cite{NeumannNorbudy},
\cite{Vistoli} and \cite{Miyanishi}, leads to the fact that if the
fibres of $P$ are rational irreducible curves, then $P$ has
 only one  horizontal curve and its fibres are isomorphic to $\C$. The fact that $P$ is a
coordinate then follows from Abhyankar-Moh-Suzuki Embedding
Theorem, as done in \cite{Miyanishi}(Lemma 1.7), or from Zariski's
main theorem, as in \cite{Vistoli}(Lemma 4.8). Instead of such
ways,  Theorem \ref{ThmNeumann} can also be proved by using the
observations below.

\medskip i) Let $H\in \C[x,y]$, $H(x,y)=\sum_{ij} c_{ij}x^iy^j$.
Recall that the so-called {\it Newton polygon} $N_H$ of $H$ is the
convex hull of the set $\{(0,0)\}\cup\{(i,j): c_{ij}\neq 0\}$.
\begin{fact} Assume that  $\deg_x H>0$ and $\deg_yH>0$. If  $H$ has only one horizontal curve,
then $N_H$ is a triangle with vertices $(0,0)$, $(\deg_xH,0)$ and
$(0,\deg_y H)$ and the summation $H_E$ of monomials $c_{ij}x^iy^j$
in $H$ with $(i,j)$ lying in the edge joining $(\deg_xH,0)$ and
$(0,\deg_yH)$ is of the form
$$H_E(x,y)=C(y^q-ax^p)^m$$
 where $C\neq 0$, $a\neq 0$, $p, q, m$ are natural numbers and $\gcd(p,q)=1$.
\end{fact}
\noindent In fact, if $H$ has only one horizontal curve, the
branches at infinity of any generic fiber $H=c$  can be given by
Newton-Puiseux expansions  of same one of the forms
$$y=bx^{p/q}+\text{lower terms in } x ,\  p/q\leq 1$$
 and $$x=by^{q/p}+\text{lower
terms in } y ,\  q/p\leq 1, $$ where $b\neq 0$ and $\gcd(p,q)=1$.
The Newton polygon $N_H$ and the face polynomial $H_E$ then can be
detected by using Newton-Puiseux theorem and the basic properties
of Newton-Puiseux expansions (see \cite{Brieskorn}).

\medskip ii) Let $P\in \C[x,y]$ with $\deg P>1$, $\deg_xP>0$ and
$\deg_yP>0$. Assume that the fibres of $P$ are irreducible
rational curves, and hence, $P$ has only one horizontal curve. In
view of Fact 1, we can assume that $P_E(x,y)=C(y^q-ax^p)^m$ with
$C\neq 0$, $a\neq 0$ and $\gcd(p,q)=1$.
\begin{fact} $P_E(x,y)$ is of the forms $C(y^q-ax)^m$ and $C(y^q-ax)^m$, $q\geq 1$.
\end{fact}
To see it, we need consider only the case $p/q\neq 1$. Assume, for
example, that $p/q <1$. Let $p:X\longrightarrow \mathbb{P}^1$ be a
regular extension over a minimal compactification $X$ of
$\mathbb{C}^2$, which results from a blowing-up process $\pi:
X\longrightarrow \P^2$ that blow up only at the indeterminacy
points of $P$ and of the obtained blowing-up versions of $P$. By
Lemma \ref{Smooth} the divisor at infinity $\DD$ is the union of
the fiber $C_\infty$ and the unique horizontal curve $D$,
$\DD=C_\infty \cup D$. The fiber $C_\infty$ contains the proper
transform $D_0$ of the line at infinity $z=0$ of $\C^2$. Since
$\deg P>1 $ and the morphism $p:X\longrightarrow \P^1$ is a
$\P^1$-fibration, the fiber $C_\infty$ is reducible and can be
contracted to one smooth rational curve by blowing down.
Furthermore, $D_0$ is the unique component of $C_\infty$ having
self-intersection $-1$. Then, one can see that the first Puiseux
chain of the dual graph of $\DD$ must be of the form
$$
\underset{-1}{\stackrel{D_0}{\bullet}}-\underset{-2}{\stackrel{e_1}{\circ}}-\cdots-
\underset{-2}{\stackrel{e_k}{\circ}}-\stackrel{|}{\underset{-2}\bigcirc}-{\stackrel{f_1}{\circ}}-\cdots-{\stackrel{f_l}{\circ}},
$$where the weights are the self-intersection numbers. This
Puiseux chain coincides with the resolution graph of the germ
curve at infinity $\gamma$, composed of the line $z=0$ and a
branch curve at infinity given by a Newton-Puiseux expansion of
the form
$$y=bx^{p/q}+\text{ lower terms in } x, \ b\neq 0.$$
Note that the line at infinity $z=0$ of $\C^2$ has
self-intersection $1$. Then,
 by examining in detail the resolving singularities of $\gamma$ we can
easily see that the self-intersection conditions $D_0^2=-1$ and
$e_i^2=-2$, $i=1,k$, are satisfied only when $p=1$, $q>1$ and
$l=1$.

\medskip Thus, once the fibres of $P$ are irreducible rational curves, Fact
2 enables us to easily construct polynomial automorphisms $\Phi$
of $\C^2$ such that $\deg P\circ \Phi(x,y)=1$.

\vskip0.3cm\noindent {\bf 4.} Let us to conclude here by a remark
that a non-zero constant Jacobian polynomial map $F=(P,Q)$ is
invertible if $P$ has only one horizontal curve. This is a key
point of the geometric proof of Theorem \ref{Razar} presented by
L\^e and Weber in \cite{LeWe2}. In fact, if $P$ has only one
horizontal curve, by the Jacobian condition in regular extensions
of $P$ the restriction of $Q$, viewed as a rational map, to the
unique horizontal curve of $P$ must be a constant mapping with
value $\infty$. So, the restriction of $Q$ to a generic fiber
$P=c$ is proper. Therefore, by the simply connectedness of $\C$
such a generic fiber $P=c$ is isomorphic to $\C$. Hence, $(P,Q)$
is invertible.

It seems to be very difficult to estimate the number of horizontal
curves of polynomial components of  non-zero constant Jacobian
polynomial maps of $\C^2$.

 \vskip 0.3cm{\it Acknowledgements.} The author would like to express his
thank to Pierete Cassou-Nogues, Mutsui Oka and Ha Huy Vui for many
valuable discussions and helps.

\end{document}